\documentclass[11pt]{article}

\usepackage{amsfonts,amsmath,amsthm,amssymb}
\usepackage{bm}
\usepackage[colorlinks, citecolor=blue, urlcolor=blue, backref=page]{hyperref}
\usepackage[round]{natbib}
\usepackage{geometry}
\usepackage[utf8]{inputenc}
\usepackage[T1]{fontenc}    
\usepackage[normalem]{ulem}
\usepackage{comment}

\usepackage{enumitem}

\newtheorem{theorem}{Theorem}[section]
\newtheorem{corollary}[theorem]{Corollary}
\newtheorem{lemma}[theorem]{Lemma}
\newtheorem{proposition}[theorem]{Proposition}

\newtheorem{example}[theorem]{Example}
\theoremstyle{definition}
\newtheorem{definition}[theorem]{Definition}

\newtheorem{remark}[theorem]{\textbf{Remark}}
\numberwithin{equation}{section}

\newcommand{\E}{{\mathbb E}}

\renewcommand{\P}{{\mathsf P}}

\newcommand{\R}{{\mathbb R}}

\newcommand{\N}{{\mathbb N}}

\newcommand{\Acal}{{\mathcal A}}

\newcommand{\Ccal}{{\mathcal C}}

\newcommand{\Ecal}{{\mathcal E}}
\newcommand{\Fcal}{{\mathcal F}}

\newcommand{\Lcal}{{\mathcal L}}

\newcommand{\Mcal}{{\mathcal M}}

\newcommand{\Pcal}{{\mathcal P}}
\newcommand{\Qcal}{{\mathcal Q}}

\DeclareMathOperator{\conv}{co}
\DeclareMathOperator{\cconv}{\overline{co}*}
\DeclareMathOperator{\cconvtv}{\overline{co}}

\DeclareMathOperator{\dist}{d}
\DeclareMathOperator{\sol}{solid}
\newcommand{\TV}{\textnormal{TV}}

\newcommand{\Peff}{\mathcal{P}_\textnormal{eff}}
\newcommand{\ba}{\textnormal{\textsf{ba}}}

\begin{document}

\title{A complete characterization of testable hypotheses}
\author{
Martin Larsson\footnote{Department of Mathematical Sciences, Carnegie Mellon University, \texttt{larsson@cmu.edu}}   \and
Aaditya Ramdas\footnote{Departments of Statistics and Machine Learning, Carnegie Mellon University \texttt{aramdas@cmu.edu}}
\and
Johannes Ruf\footnote{Department of Mathematics, London School of Economics, \texttt{j.ruf@lse.ac.uk}}\\[2ex]
}
\maketitle

\begin{abstract}
We revisit a fundamental question in hypothesis testing: given two sets of probability measures $\mathcal{P}$ and $\mathcal{Q}$, when does a nontrivial (i.e.\ strictly unbiased) test for $\mathcal{P}$ against $\mathcal{Q}$ exist? Le~Cam showed that, when $\mathcal{P}$ and $\mathcal{Q}$ have a common dominating measure, a test that has power exceeding its level by more than $\varepsilon$ exists if and only if the convex hulls of $\mathcal{P}$ and $\mathcal{Q}$ are separated in total-variation distance by more than $\varepsilon$. The requirement of a dominating measure is frequently violated in nonparametric statistics.
In a passing remark, Le~Cam described an approach to address more general scenarios, but he stopped short of stating a formal theorem.
This work completes Le~Cam's program, by presenting a matching necessary and sufficient condition for testability: for the aforementioned theorem to hold without assumptions, one must take the closures of the convex hulls of $\mathcal{P}$ and $\mathcal{Q}$ in the space of bounded finitely additive measures. 
We provide simple elucidating examples, and elaborate on various subtle measure-theoretic and topological points regarding compactness and achievability. 
\end{abstract}

\section{Introduction} \label{sec:intro}

Let $\Mcal_1$ denote the set of all (countably additive) probability measures  on a given measurable space $(\Omega,\Fcal)$. The null and alternative hypotheses are $\Pcal,\Qcal \subset \Mcal_1$, always assumed nonempty. A \emph{test} $\phi$ is a $[0,1]$-valued measurable function on this space, and its worst-case level (or type-I error) and power are given by $\sup_{\mu \in \Pcal} \E_\mu[\phi]$ and $\inf_{\nu \in \Qcal} \E_\nu[\phi]$. Let $\Phi$ denote the set of all tests. One calls a test $\phi$  \emph{strictly unbiased} if $$\sup_{\mu \in \Pcal} \E_\mu[\phi] < \inf_{\nu \in \Qcal} \E_\nu[\phi],$$
meaning that its worst-case power is larger than its worst-case level. Since this is a rather reasonable minimal requirement for a test, we use the shorter term ``nontrivial'' to mean strictly unbiased. This paper answers a simple question: when does a nontrivial test exist? 

The risk of a test $\phi$ is defined as the sum of its worst-case type-I and type-II errors,
\[
 R(\phi) = R(\phi,\Pcal,\Qcal) = \sup_{\mu \in \Pcal} \E_\mu[\phi] + \sup_{\nu \in \Qcal} \E_\nu[1-\phi].
\]
The risk is unaffected if $\Pcal$ and $\Qcal$ are replaced by their convex hulls $\conv(\Pcal)$ and $\conv(\Qcal)$, which consist of all finite convex combinations of elements of $\Pcal$ and $\Qcal$. That is,
 \[
 R(\phi,\Pcal,\Qcal) = R(\phi,\conv(\Pcal),\conv(\Qcal)).
 \]
The minimax risk is defined by
\[
R(\Pcal,\Qcal) = \inf_{\phi \in \Phi} R(\phi, \Pcal, \Qcal).
\]
Clearly a test $\phi$ is nontrivial if and only if $R(\phi) < 1$, and such a test exists if and only if $R(\Pcal,\Qcal) < 1$.
If $R(\Pcal,\Qcal) = R(\phi^*)$ for some $\phi^* \in \Phi$, we call $\phi^*$ a minimax optimal test. 
Because the risk $R(\phi)$ is a supremum of affine functions, it is convex. Thus solving for the minimax risk is a convex minimization problem. However, $\Phi$ is not in general compact in any topology for which the risk is continuous (or lower semicontinuous), and so a minimax optimal test may not always exist.
It is well-known that a minimax optimal test \emph{does} exist under certain conditions; this is discussed further below.

The minimax risk is closely related to the total-variation (TV) distance between $\conv(\Pcal)$ and $\conv(\Qcal)$.
Indeed, for any test $\phi$ and arbitrary 
probability measures $\mu,\nu$, it  holds that
\begin{equation*}
\E_\mu[\phi] + \E_\nu[1-\phi] \geq 1-\dist_\TV(\mu,\nu), 
\end{equation*}
where $\dist_\TV$ denotes the TV distance; see~\eqref{eq:TV-dist-formula-proof-section} for details.
As a result we have
\[
R(\phi) \geq 1-\dist_\TV(\mu,\nu)
\]
for any $\phi \in \Phi$, $\mu \in \conv(\Pcal)$, and $\nu \in \conv(\Qcal)$. Minimizing over $\phi$ and maximizing over $\mu$ and $\nu$ then yields the ``weak duality''
\begin{equation} \label{eqn:weak-duality-intro}
R(\Pcal, \Qcal) \geq 1-\dist_\TV(\conv(\Pcal), \conv(\Qcal)),
\end{equation}
where the TV distance between sets $S,T$ of (finitely or countably additive) probability measures is given by $\dist_\TV(S,T) = \inf_{\mu \in S,\, \nu \in T} d_\TV(\mu,\nu)$.
To foreshadow later developments, we note that the definition of TV distance carries over without change to (sets of) \emph{finitely additive} probability measures.

Because the distance between two sets is always the same as the distance between their closures, we could replace $\conv(\Pcal)$ and $\conv(\Qcal)$ by their TV closures $\cconvtv(\Pcal)$ and $\cconvtv(\Qcal)$ in \eqref{eqn:weak-duality-intro}.
Crediting Lucien Le~Cam, \cite{kraft1955some} presented the following fundamental result, which is a sufficient condition for equality to hold in \eqref{eqn:weak-duality-intro}.

\begin{theorem} \label{T:classical}
If $\Pcal$ and $\Qcal$ have a common dominating measure $\gamma$ (meaning that $\mu \ll \gamma$ and $\nu \ll \gamma$ for all $\mu\in\Pcal$, $\nu\in\Qcal$, {\color{black}and always understood to be $\sigma$-finite}), then for any $\varepsilon \geq 0$,
{\color{black}
\begin{equation}\label{eq:kraft-lecam}
\begin{aligned}
\exists \text{ test } \phi \in \Phi \colon \inf_{\nu \in \Qcal} \E_\nu&[\phi] > \sup_{\mu \in \Pcal} \E_\mu[\phi]  + \varepsilon \quad \\
&\iff \quad \dist_\TV(\cconvtv(\Pcal), \cconvtv(\Qcal)) > \varepsilon.
\end{aligned}
\end{equation}
}
In fact,  we have
\begin{equation}\label{eq:minimax-risk-classical}
R(\Pcal,\Qcal) = 1- \dist_\TV(\cconvtv(\Pcal), \cconvtv(\Qcal))  
\end{equation}
where the minimax risk is achieved by some $\phi^* \in \Phi$, but the infimum in $d_\TV$ is in general not achieved by any $\mu \in \cconvtv(\Pcal),\nu \in \cconvtv(\Qcal)$.
\end{theorem}

In words, if $\Pcal$ and $\Qcal$ have a common dominating measure, then a nontrivial test exists if and only if the convex hulls of $\Pcal$ and $\Qcal$ are separated in TV distance. Actually, \cite{kraft1955some} presented only~\eqref{eq:kraft-lecam} and proved it using the Hahn--Banach separating hyperplane theorem. A statement like \eqref{eq:minimax-risk-classical} appears on page~476 of~\cite{le2012asymptotic} and was proved succinctly using a minimax theorem, but the left-hand side there is subtly (yet critically) different; we return to Le Cam's claim in Theorem~\ref{th:gen-test} and  again in Appendix~\ref{sec:gen-tests}. Also using Hahn--Banach, Proposition~\ref{prop:ref-meas-tv-dist} shows that  Theorem~\ref{T:classical} is a corollary of our own more general Theorem~\ref{T_kraft}, itself proved using a minimax theorem.

\begin{example}
\label{ex:johannes-beats-gemini}
The infimum TV distance is not achieved in general in Theorem~\ref{T:classical}. To see this, let $\Qcal = \{\delta_0\}$ and 
$$\Pcal = \left\{\left(\frac12 - \frac1n\right)\delta_0 + \left(\frac12 + \frac1n\right)\delta_n \colon n \in \N,\ n \ge 3\right\}.$$ 
Then $\dist_\TV(\conv(\Pcal), \delta_0)=1/2$, but this is not achieved by any $\mu \in \cconvtv(\Pcal)$ because there is no such $\mu$ with mass exactly 1/2 at 0. Said differently, the TV projection onto $\cconvtv(\Pcal)$ may not exist. (The TV projection always exists onto $\cconv(\Pcal)$, an object introduced later.)
\end{example}

The assumption of a common dominating measure made in the above theorem appears innocuous but is violated in many standard nonparametric testing problems. For example, here are commonly encountered classes which do not have a dominating measure: 
\begin{enumerate}
    \item distributions on a bounded support (like $[0,1]$) with a specified mean (like 0.5);
    \item all symmetric distributions around the origin;
    \item all (unbounded) distributions with a specified mean and bound on the variance;
    \item  all $\sigma$-sub-Gaussian distributions for a given $\sigma > 0$;
    \item the TV, Wasserstein, or Kolmogorov ball around a given distribution.
\end{enumerate}

In all these cases, Theorem~\ref{T:classical} is silent. The importance of studying such settings has not escaped statisticians. For example, ~\cite{Huber:Strassen:1973} remark at the end of their Section 6 that ``...the more interesting sets $\Pcal,\Qcal$ are not dominated...''.
Here is a very simple example from~\cite{van2002statistical} where the assumption of Theorem~\ref{T:classical} fails and its conclusion also fails. 

\begin{example}\label{eg:van-de-waart}
Let $\Pcal =\{\delta_x: x \in [0,1]\}$, where $\delta_x$ is the Dirac delta mass at $x$, and let $\Qcal=\{\mathrm{Uni}[0,1]\}$, where $\mathrm{Uni}[0,1]$ denotes the uniform distribution on $[0,1]$. The TV distance between $\mathrm{Uni}[0,1]$ and $\cconvtv(\Pcal)$ equals 1, the largest possible value. 
However, as should be intuitively clear, and also follows as a consequence of our main theorem below, no nontrivial test exists for this problem and the minimax risk is $R(\Pcal, \Qcal) = 1$. (See also Proposition~\ref{P:2}.) {\color{black}Note that above, a test $\phi$ maps a \emph{single observation} to $[0,1]$.} Letting $\Pcal=\{\delta_x \times \delta_x: x\in[0,1]\}$ and $\Qcal = \{\mathrm{Uni}[0,1] \times \mathrm{Uni}[0,1]\}$, one can design a perfect test that has zero type-I error and power one (the test rejects if the two observations are different). 
While our setup allows the underlying $\Omega$ and distributions to be arbitrary, and covers the case of $\Pcal$ and $\Qcal$ consisting of product distributions as a special case, the purpose of the example is to show that Theorem~\ref{T:classical} can fail in nondominated settings.  
\end{example}

After seeing the above example, it may be tempting to conjecture that when there is no reference measure, the closure in TV may need to be replaced by the \emph{standard weak closure} (i.e.\ closure in the standard topology of weak convergence of probability measures). Indeed in the above example, the weak closure of $\conv(\Pcal)$ contains $\mathrm{Uni}[0,1]$. However, the following simple counterexample  shows that using the weak closure of the convex hull is also in general the wrong answer.

\begin{example} \label{ex:intro-2}
Take $\Omega = [0,1]$, $\Pcal=\{\delta_x \colon x\in[0,0.5)\}$, and $\Qcal=\{\delta_y \colon y\in(0.5,1]\}$. Then a perfect test exists: $\bm1_{\{X>0.5\}}$ has level 0 and power 1. However, the standard weak closures of both $\Pcal$ and $\Qcal$ include $\delta_{0.5}$. In short, perfect tests exist even though the weak closures of the null and alternative hypotheses intersect, so~\eqref{eq:kraft-lecam} could not have used the weak closures of $\conv(\Pcal)$ and $\conv(\Qcal)$, respectively. 
Once more, our theorem will lead to the right conclusion.
\end{example}

In the first example above, the TV closure was not large enough. In the second example, the standard weak closure was too large. 
Our main result employs a type of closure that lies in between (in strength) the standard weak and strong closures: 
we take the closure in the \emph{weak-$*$ topology on $\ba$}, the space of bounded finitely additive measures.
{\color{black}In this topology, convergence of $\mu_\alpha$ to $\mu$ is equivalent to convergence of $\E_{\mu_\alpha}[f]$ to $\E_\mu[f]$ for all bounded measurable functions $f$.}
The formal definitions are reviewed further in Section~\ref{sec:ba-and-proofs}.

Let us give some intuition for why the weak-$*$ topology arises naturally. 
Suppose a test $\phi$ has risk $R(\phi) = r$, so that $\E_\mu[\phi] + \E_\nu[1-\phi] \le r$ for all $\mu \in \Pcal$ and $\nu \in \Qcal$, and $r$ is the smallest number with this property. 
This inequality remains true for all $\mu, \nu$ in the convex hulls $\conv(\Pcal), \conv(\Qcal)$, and then for all $\mu, \nu$ in the \emph{closures} of the convex hulls---assuming the closures are taken in a topology for which the expectation $\E_\mu[\phi]$ is continuous in $\mu$. 
The significance of the weak-$*$ topology on $\ba$ is that it is the \emph{weakest} topology with this property for any test $\phi$. Thus the weak-$*$ convex closures $\cconv(\Pcal), \cconv(\Qcal)$ are the \emph{largest} extensions of $\Pcal, \Qcal$ that preserve the risk of any test $\phi$. It is the distance between these closures that matters for the minimal achievable risk.
While this motivates the use of the weak-$*$ topology, it does not explain why the space $\ba$ is needed. This is indeed a more subtle matter, which we address further in Section~\ref{sec:ba-and-proofs}.

We now state our main result giving a quantitative characterization of testable hypotheses.
(We recall later in \eqref{eq:TV-dist-formula-proof-section} how the TV distance is defined in $\ba$.)

\begin{theorem} \label{T_kraft}
For any $\Pcal,\Qcal\subset \Mcal_1$ and $\varepsilon \ge 0$, we have the following equivalence:
{\color{black}
\begin{align*}
\exists \text{ test } \phi : \inf_{\nu \in \Qcal} \E_\nu&[\phi] >  \sup_{\mu \in \Pcal} \E_\mu[\phi] + \varepsilon \\
&\iff 
\dist_\TV(\cconv(\Pcal), \cconv(\Qcal)) > \varepsilon.
\end{align*}
}
In fact, we have
\begin{equation}\label{eq:minimax-risk-ba}
  R(\Pcal,\Qcal) = 1- \dist_\TV(\cconv(\Pcal), \cconv(\Qcal)),
\end{equation}
where 
the infimum in the $d_\TV$ is achieved by some $\mu \in \cconv(\Pcal), \nu \in \cconv(\Qcal)$.
\end{theorem}
We suspect that outside of special cases, such as the dominated setting of Theorem~\ref{T:classical}, the minimax risk may not in general be achieved by any test $\phi$.

The proof of the above result is based on a minimax theorem, as detailed in the following section.
We note that the choice of topology is a critical and subtle matter in this regard, because it affects both the compactness of sets and continuity of functions, required for invoking an appropriate minimax theorem. In Appendix~\ref{app:topology} we expand briefly on this, 
noting that the established literature contains instances where minimax theorems are applied incorrectly.

While Theorem~\ref{T_kraft} holds under weaker assumptions, it is not immediately apparent that it implies Theorem~\ref{T:classical}. The following result clarifies this.

\begin{proposition} \label{prop:ref-meas-tv-dist}
If $\Pcal$ and $\Qcal$ admit a common reference measure $\gamma$, then 
\[
\dist_\TV(\cconvtv(\Pcal), \cconvtv(\Qcal)) = \dist_\TV(\cconv(\Pcal), \cconv(\Qcal)).
\]
Moreover, the minimax risk is always attained by a minimax optimal test, so Theorem~\ref{T_kraft} implies~Theorem~\ref{T:classical}. 
\end{proposition}

The above result holds despite the fact that $\cconvtv(\Pcal) \neq \cconv(\Pcal)$ in general.  
{\color{black} To see that the two closures can be different, we consider the sample space $\N$, put
$\Pcal=\{\delta_n \colon n \in \N\}$, and note that $\cconv(\Pcal)$, which is weak-$*$ compact, must include some limit point of the sequence of point masses $\delta_n$. 
We pick such a limit point and denote it by $\delta_\infty$ because it is ``supported at infinity'', in that it assigns zero mass to any finite subset of $\N$. In particular, this implies that $\delta_\infty$ is purely finitely additive since any nonzero countably additive component would have to charge some finite subsets of $\N$.} (We may also note this from Example~\ref{ex:johannes-beats-gemini}, where $\cconv(\Pcal)$ contains an equal mixture of $\delta_0$ and $\delta_\infty$.)

Thus, Theorem~\ref{T_kraft} appears to be fundamental since it holds without any (statistically relevant) restrictions, and is the main contribution of this paper. We also present another result that is of independent interest, because it involves the usual topology of weak convergence and does not involve the weak-$*$ closure in $\ba$.

{\color{black}
\begin{remark}
Recall that if $\Omega$ is a metric space and $C_b$ denotes the set of continuous bounded functions on $\Omega$, the \emph{topology of weak convergence} on $\Mcal_1$ is the weakest one for which $\mu \mapsto \E_\mu[f]$ is continuous for every $f \in C_b$. Equivalently, a sequence $\mu_n$ converges weakly to $\mu$ if $\E_{\mu_n}[f]$ converges to $\E_\mu[f]$ for every $f \in C_b$.
Thanks to Prokhorov's theorem (see \citet[Theorem~5.1]{MR1700749}) one can check weak compactness by checking weak closedness and tightness.
This topology differs from the weak-$*$ topology used above both in the choice of test functions (continuous instead of measurable) and because the underlying set is different ($\Mcal_1$ instead of $\ba$); see Section~\ref{sec:ba-and-proofs}.
\end{remark}
}

\begin{theorem}\label{th:weakly-compact}
Assume $\Omega$ is a metric space and $\Fcal$ the Borel $\sigma$-algebra.
If $\Pcal$ and $\Qcal$ are convex and weakly compact (in the usual sense), then
{\color{black}
\begin{align*}
R(\Pcal,\Qcal) &= \inf_{\phi \in \Phi \cap C_b} R(\phi) \\
&= 1 - \dist_\TV(\Pcal, \Qcal) \\
&= 1 - \dist_\TV(\cconv(\Pcal), \cconv(\Qcal)).
\end{align*}
}
In particular, \eqref{eq:kraft-lecam} holds and the test $\phi$ can be taken to be continuous therein, and 
the infimum in the total-variation distance $d_\TV$ is achieved by some $\mu \in \Pcal, \nu \in \Qcal$. 
\end{theorem}

One can also rephrase the above result in the following manner: if $\conv(\Pcal)$ and $\conv(\Qcal)$ are weakly compact, then $R(\Pcal,\Qcal) =  1 - \dist_\TV(\conv(\Pcal), \conv(\Qcal))$.
Note also that above, $\Pcal$ and $\Qcal$ need not have a reference measure (Example~\ref{ex:bounded-means})
and so Theorem~\ref{th:weakly-compact} itself is incomparable to Theorem~\ref{T:classical}.

\begin{remark}
In the context of Theorem~\ref{th:weakly-compact},
    let us first note that in general, $\Pcal \subsetneq \cconv(\Pcal)$. Consider $\Omega=[-1,1]$ and $\Pcal$ to be the set of all symmetric laws around 0, which is convex and weakly compact. Now, consider the sequence $\mu_n = \mathrm{Uni}[-1/n,1/n]$, which converges weakly to $\delta_0$, but does not converge to $\delta_0$ in the weak-$*$ topology. (Indeed, the test function $\phi(x)=0$ for $x \neq 0$ and $\phi(x)=1$ for $x=0$ satisfies $\E_{\mu_n}[\phi]=0$, but $\E_{\delta_0}[\phi]=1$, and such discontinuous test functions are permissible in the weak-$*$ topology.) {\color{black}However, by compactness the sequence $\mu_n$ does have a subnet that converges to an element $\mu_\infty$ in $\ba$. This limit assigns mass zero to $\{0\}$ but mass one to any open interval containing $0$. In particular, this shows that $\mu_\infty \notin \Mcal_1$ and that $\Pcal \subsetneq \cconv(\Pcal)$ as claimed.} Despite this, the TV distance in the above theorem is unaffected by taking the closure in $\ba$ because in a metric space with the Borel $\sigma$-algebra, $\dist_\TV(\mu,\nu) =   \sup_{\phi \in \Phi \cap C_b} (\E_\mu[\phi] - \E_\nu[\phi])$; see Lemma~\ref{lem:TV-Cb}. That is, we only need to consider  \emph{continuous} $[0,1]$-bounded functions to compute the TV distance. In fact, Lemma~\ref{lem:TV-Cb} is the only reason we need a metric, rather than merely a topology, in Theorem~\ref{th:weakly-compact}.
\end{remark}

The weak-$*$ closure is always at least as large as the TV closure, so the right-hand side of \eqref{eq:minimax-risk-ba} is always at least as large as the right-hand side of \eqref{eq:minimax-risk-classical}. Moreover,
there exist problems (like Example~\ref{eg:van-de-waart}) where
\[
 R(\Pcal,\Qcal) > 1 - \dist_\TV(\cconvtv(\Pcal), \cconvtv(\Qcal)).
\]
Hence, we see that
\eqref{eq:minimax-risk-classical} is overly optimistic about the minimax risk, in the sense that problems without a reference measure can be harder than what \eqref{eq:minimax-risk-classical}  may suggest. Even though we have not shown that a minimax optimal test always exists, the following  result is still practically useful to verify that a candidate test is, in fact, minimax optimal.

\begin{corollary}\label{cor:saddle}
For any $\Pcal,\Qcal \subset \Mcal_1$, and any $\phi \in \Phi, \mu \in \cconv(\Pcal),\nu \in \cconv(\Qcal)$, one has the weak duality inequality
\[
R(\phi) \geq 1-\dist_\TV(\mu,\nu).
\]
Thus, if one can exhibit $\phi^* \in \Phi,\mu^* \in \cconv(\Pcal),\nu^* \in \cconv(\Qcal)$ such that
\[
R(\phi^*) = 1 - \dist_\TV(\mu^*,\nu^*),
\]
then $\phi^*$ is minimax optimal, and $\mu^*,\nu^*$ are the TV-closest pair  between $\cconv(\Pcal), \cconv(\Qcal)$.
\end{corollary}

\begin{proof}
The first display is an immediate consequence of~\eqref{eq:minimax-risk-ba}. The second display implies
$
R(\Pcal,\Qcal) \leq R(\phi^*) = 1 - \dist_\TV(\mu^*,\nu^*) \leq 1 - \dist_\TV(\cconv(\Pcal), \cconv(\Qcal)),
$
but since the first and last quantities are equal by~\eqref{eq:minimax-risk-ba}, equality must hold throughout.
\end{proof}

An interesting nonparametric example where the above corollary proves particularly useful is the following.

\begin{example}\label{ex:bounded-means}
Let $\Pcal$ be all probability measures on $[0,1]$ whose mean is at most $m_1$, and let $\Qcal$ be those with mean at least $m_2$, where $0 < m_1 < m_2 < 1$. These sets do not have a dominating reference measure, but are weakly compact and convex, and so Theorem~\ref{th:weakly-compact} applies, but we can say more. 
Consider the test $\phi^*(x)=x$ and let $\mu^*,\nu^*$ be Bernoullis with means $m_1,m_2$. Then $R(\phi^*)$ is easily seen to equal $m_1 + (1-m_2)$, which equals $1-\dist_\TV(\mu^*,\nu^*)$. Thus, Corollary~\ref{cor:saddle} implies that $\phi^*$ is minimax optimal and the two Bernoullis are the TV-closest pair.
Thus, here, the minimax optimal test exists, despite Theorem~\ref{th:weakly-compact} not guaranteeing its existence.
\end{example}

{\color{black}There are other recent attempts we are aware of to handle nondominated hypotheses. For instance,~\cite{liebrich2022model} consider the class of so-called \emph{supported} $\Pcal,\Qcal$, and following our first preprint,~\citet[Theorem~52]{klein2026can} consider \emph{pre-Hahn localizable} $\Pcal,\Qcal$ (with some further assumptions). These works aim to restrict $\Pcal,\Qcal$ sufficiently so that, while being nondominated, the corresponding strong duality  does not need to consider finitely additive measures. We note, of course, that Theorem~\ref{th:weakly-compact} also achieves this goal.}
The last result that we wish to highlight is that of~\cite{Huber:Strassen:1973}, who prove that if the Choquet capacities induced by $\Pcal,\Qcal$ are \emph{two-alternating}, then a \emph{least favorable distribution pair} exists, such that a likelihood ratio test between them is minimax optimal. This result applies to both parametric settings which may not be compact (e.g., testing unit variance Gaussians with nonpositive mean against those with mean larger than one), as well as nonparametric settings without a reference measure (e.g., when the null and alternative are each described by a total-variation ball around a simple hypothesis). Nevertheless, it does not yield a necessary and sufficient condition for testability as Theorem~\ref{T_kraft} does, just a (very useful) sufficient one.

As a final note, we highlight that Le~Cam appeared to be fully aware of the deficiencies of Theorem~\ref{T:classical}. However, his preferred solution was highly nonstandard. 
Notice that Theorem~\ref{T_kraft} alters the right-hand side of  \eqref{eq:kraft-lecam} but leaves the left-hand side unchanged. Le~Cam's book~\citeyearpar[pg.\ 476]{le2012asymptotic} presents a statement that changes the left-hand side of  \eqref{eq:kraft-lecam}, but leaves the right-hand side unchanged. However, while mathematically correct, its statistical implications are limited, as also pointed out by~\cite{van2002statistical}.
We nevertheless provide it below for completeness.

Write $\Mcal$ for the set of finite signed countably additive measures on $\Omega$, which is a Banach space when equipped with the TV-norm, and let $\Mcal'$ denote its dual {\color{black}space of continuous linear functionals}.
A \emph{generalized test} is an element $\phi \in \Mcal'$ such that $\phi(\mu) \in [0,1]$ for all $\mu \in \Mcal_1$. (The reader is warned that Le~Cam just called these \emph{tests}, which may lead to confusion unless one very carefully follows his definitions.)

\begin{theorem}[\cite{le2012asymptotic}]\label{th:gen-test}
Let $\Pcal$ and $\Qcal$ be arbitrary subsets of $\Mcal_1$, and let $\varepsilon \ge 0$. Then we have the following equivalence:
{\color{black}
\begin{align*}
\exists \text{ generalized test } \phi: & \inf_{\nu \in \Qcal}  \phi(\nu) > \sup_{\mu \in \Pcal} \phi(\mu)  + \varepsilon \\
&\iff \dist_\TV(\conv(\Pcal), \conv(\Qcal)) > \varepsilon.
\end{align*}
}
\end{theorem}

As Le~Cam had also noted, these generalized tests may not correspond to bounded measurable functions, and thus need not be possible to evaluate using observed samples{\color{black}; certain additional information may be required}. The tests $\phi$ that we consider also live in $\mathcal{M}'$, but are in fact given by bounded measurable functions (by definition). Thus the above result, while correct, has limited practical applicability, unlike our main Theorem~\ref{T_kraft}. We discuss this point in detail in Appendix~\ref{sec:gen-tests}.

Once more, Le~Cam was aware of this drawback of his mathematically general (but statistically opaque) formulation in Theorem~\ref{th:gen-test}. In a remark following Lemma~1 on page~476 of~\cite{le2012asymptotic}, he even suggested a way of only working with tests that are bounded measurable functions. Following his sentence ``\emph{One case which occurs often in the study of robust procedures
is as follows...}'', he goes on to describe a proof strategy for a result similar to Theorem~\ref{T_kraft}. However, unfortunately, Le~Cam only treats this in a passing remark, that is easily missed (or unappreciated) due to its abstractness. A clearly stated theorem, which actually holds with no unnecessary restrictions on the tests or hypotheses (like Theorem~\ref{T_kraft}), was not provided. It remains unclear to us whether Le~Cam did not realize that what he was proposing in that remark were the seeds of a complete solution, or whether he realized it but simply chose not to emphasize it (despite its apparent importance) for reasons of personal preference. Such ponderings aside, we believe it to be the case that most statisticians, even after parsing his remark, would not be able to clearly  answer the question posed at the start of the abstract, motivating us to write this paper to settle the issue.

We end the introduction by noting that there is much precedent for the use of finitely additive measures in statistics, dating back nearly a century. An early proponent was de~Finetti (circa 1930), whose push for finite additivity in subjective Bayesian probability continues to be the topic of much discussion today; see~\cite{bingham2010finite,regazzini2013origins,seidenfeld2025finite} and references therein. Another prominent book to employ finite additivity is that of~\cite{dubins2014gamble}, who claim that many of their theorems are more general and easier to prove when working with finitely additive measures. 

We find it critical to point out that this present paper is very different from the above in its use of finitely additive measures. Our paper does not seek to promote or defend the use of finite additivity as an axiom or mathematical tool of convenience. We do not take a position on this issue. 
Our mathematical formalism and setup, from questions asked to assumptions made, are about countably additive probability measures, as the currently most prevalent model in probability and statistics. 
We did make a significant effort to stay in the realm of countable additivity, but found that there is no escape: {\color{black} least favorable distribution pairs that witness the minimax risk can in fact have nonzero finitely additive components;} 
see e.g.\ Example~\ref{ex:251212}.
Thus finite additivity appears as a fundamental and unavoidable mathematical consequence of providing a complete answer to a question about countably additive probability, very different from the aforementioned works of de~Finetti, Dubins--Savage, and others.

The rest of this paper is organized as follows. 
    Section~\ref{sec:ba-and-proofs} presents the mathematical background to understand the \emph{weak-$*$ closure in $\ba$}, and the proofs of all the results stated above.
    Section~\ref{sec:singleton} discusses the interesting case of a singleton $\Qcal=\{\nu\}$ and its relation to recent work.
    Appendix~\ref{app:topology} elaborates on some subtle issues related to the choice of topology.
    Appendix~\ref{sec:gen-tests} discusses Le Cam's generalized tests in more detail.
Appendix~\ref{A:measure-theoretic} contains some auxiliary technical results.
Appendix~\ref{S_proof_T_Peff} proves Theorem~\ref{T_Peff}, which is a result relating the closed convex hull in $\ba$ to the \emph{effective null} of $\Pcal$, another fundamental object from recent work.
{\color{black}Lastly, Appendix~\ref{app:practical-implications} outlines some consequences of our work for statistical applications, including a convex programming representation of the minimax risk.}

\section{Finite additivity and main proofs} \label{sec:ba-and-proofs}

Our goal here is to prove the results presented in Section~\ref{sec:intro} and, in the process of doing so, review the required mathematical machinery and elucidate the role of finitely additive measures. General references on this topic include \cite{MR45194,MR751777}.
Recall that we have fixed a measurable space $(\Omega, \Fcal)$, and {\color{black}we let $\Lcal$ denote the set of bounded measurable functions on this space.}

\begin{definition}
\label{def:tv}
A \emph{bounded finitely additive measure} is a real-valued set function $\mu$ on $\Fcal$ that is (i) \emph{additive}, $\mu(A \cup B) = \mu(A) + \mu(B)$ for all disjoint sets $A, B \in \Fcal$, and (ii) \emph{bounded} in total-variation norm, 
$\|\mu\|_\TV = \sup_\pi \sum_{A \in \pi} |\mu(A)| < \infty$, where the supremum ranges over all finite measurable partitions of $\Omega$.
The space of all such $\mu$ is denoted by $\ba$, the subset taking nonnegative values by $\ba_+$, and the further subset with unit mass by $\ba_1$. Thus $\ba_1$ is the space of finitely additive probability measures, also known as \emph{probability charges}.
\end{definition}

Since countable additivity implies finite additivity, one has $\Mcal_1 \subset \ba_1$, and the inclusion is strict as soon as $\Fcal$ is infinite.
Every $\mu \in \ba_+$ admits a unique decomposition $\mu = \mu_c + \mu_p$ into a countably additive part $\mu_c \in \Mcal_+$ (the set of nonnegative measures in $\Mcal$) and a \emph{purely finitely additive} part $\mu_p \in \ba_+$. That $\mu_p$ is purely finitely additive means that any countably additive $\nu$ such that $0 \le \nu \le \mu_p$ (setwise) must equal 0. 

The integral $\int_\Omega f d\mu$, or $\E_\mu[f]$ in probabilistic notation, is defined for any $f \in \Lcal$ and any $\mu \in \ba$, and is linear in $f$ and $\mu${\color{black}; see \citet[Section~2]{MR45194}}. Limit theorems such as the dominated and monotone convergence theorems need not hold for finitely additive integration, but one does have the bound $|\E_\mu[f]| \le \|f\|_\infty \|\mu\|_\textnormal{TV}$, where $\|f\|_\infty = \sup_{\omega \in \Omega} |f(\omega)|$ is the supremum norm.
This leads to the identification of $\ba$ as the dual of the Banach space $(\Lcal,\|\cdot\|_\infty)$.
 
Like the dual of any Banach space, $\ba$ can be equipped with the \emph{weak-$*$ topology} $\sigma(\ba,\Lcal)$, defined as the weakest topology such that the linear functionals $\mu \mapsto \E_\mu[f]$ are continuous for all $f \in \Lcal$.
{\color{black}Equivalently, a net $\mu_\alpha$ converges to $\mu$ in the weak-$*$ topology if $\E_{\mu_\alpha}[f]$ converges to $\E_\mu[f]$ for every $f \in \Lcal$.
We write $\cconv(\Pcal)$ for the weak-$*$ closure in $\ba$ of $\conv(\Pcal)$, and similarly for other sets.} 
The Banach--Alaoglu theorem implies that the set $\ba_1$ is weak-$*$ compact, and then so is any weak-$*$ closed subset of $\ba_1$. In particular, this is the case for $\cconv(\Pcal)$ and $\cconv(\Qcal)$.
It is this compactness property under the weak-$*$ topology that explains the appearance of the space $\ba$ in Theorem~\ref{T_kraft}.

The compactness is crucial because the proof of Theorem~\ref{T_kraft} relies on the following special case of a minimax theorem due to 
\citet[Theorem~2]{Fan:1953}.

\begin{theorem}
Let $X$ be a compact convex subset of a Hausdorff topological vector space, and $Y$ a convex subset of a vector space. Let $F$ be a real-valued function on $X \times Y$ such that $F(x,y)$ is lower semicontinuous and convex in $x$ for each fixed $y \in Y$, and concave in $y$ for each fixed $x \in X$. Then
\[
\inf_{x \in X} \sup_{y \in Y} F(x,y) = \sup_{y \in Y} \inf_{x \in X} F(x,y),
\]
where the infima are attained thanks to compactness of $X$ and lower semicontinuity in $x$ of $F(x,y)$ and $\sup_{y \in Y} F(x,y)$.
\end{theorem}

The total-variation norm on $\ba$ has the representation $\|\mu\|_\textnormal{TV} = \sup_{f \in B_1} \E_\mu[f]$, where $B_1$ is the unit ball in $\Lcal$.
It is customary to define the total-variation distance $\dist_\textnormal{TV}(\mu,\nu)$ between two probability charges $\mu,\nu$ as \emph{half} the total-variation norm of $\mu-\nu$.
Because $\mu,\nu$ have unit mass, and because $B_1 = 2\Phi - 1$, meaning that $f \in B_1$ if and only if $f = 2\phi - 1$ for some test $\phi$, one obtains the representation
\begin{equation} \label{eq:TV-dist-formula-proof-section}
\dist_\textnormal{TV}(\mu,\nu) = \frac12 \|\mu - \nu\|_\textnormal{TV} = \sup_{\phi \in \Phi} \big(\E_\mu[\phi] - \E_\nu[\phi]\big), \quad  \mu,\nu \in \ba_1.
\end{equation}

With these preliminaries in hand, the proof of Theorem~\ref{T_kraft} is straightforward.

\begin{proof}[Proof of Theorem~\ref{T_kraft}]
We will apply Fan's minimax theorem with $X = \cconv(\Pcal) \times \cconv(\Qcal)$, $Y = \Phi$, and $F \colon ((\mu,\nu),\phi) \mapsto \E_\nu[\phi] - \E_\mu[\phi]$.
Indeed, $X$ is a compact convex subset of the product space $\ba \times \ba$, where each factor has the weak-$*$ topology, $Y$ is a convex subset of $\Lcal$, and $F$ has the required concavity and convexity properties and is continuous in $(\mu,\nu)$ by definition of the weak-$*$ topology.
Using first Fan's minimax theorem, then linearity and weak-$*$ continuity of $\mu \mapsto \E_\mu[\phi]$, and finally a simple algebraic manipulation, we have
\begin{align*}
\inf_{\substack{\mu \in \cconv(\Pcal) \\ \nu \in \cconv(\Qcal)}} \sup_{\phi \in \Phi} \big( & \E_\nu[\phi] - \E_\mu[\phi] \big) \\
&= \sup_{\phi \in \Phi} \inf_{\substack{\mu \in \cconv(\Pcal) \\ \nu \in \cconv(\Qcal)}} \big( \E_\nu[\phi] - \E_\mu[\phi] \big) \\
&= \sup_{\phi \in \Phi} \inf_{\substack{\mu \in \Pcal \\ \nu \in \Qcal}} \big( \E_\nu[\phi] - \E_\mu[\phi] \big) \\
&= 1 - \inf_{\phi \in \Phi} \sup_{\substack{\mu \in \Pcal \\ \nu \in \Qcal}} \big( \E_\mu[\phi] + \E_\nu[1-\phi] \big),
\end{align*}
where the infimum on the left-hand side is attained by some $\mu^* \in \cconv(\Pcal)$ and $\nu^* \in \cconv(\Qcal)$.
This is the desired statement, because the left-hand side is equal to $\dist_\TV(\cconv(\Pcal), \cconv(\Qcal))$ in view of \eqref{eq:TV-dist-formula-proof-section}, while the right-hand side is equal to $1-R(\Pcal,\Qcal)$.
\end{proof}

A similar, but different, calculation can be used to obtain Theorem~\ref{th:weakly-compact}.
This calculation will make use of the assumption that $\Pcal$ and $\Qcal$ are convex and weakly compact in the usual sense, i.e., in the topology $\sigma(\Mcal_1, C_b)$.
By considering \emph{continuous} tests $\phi \in \Phi_c = \Phi \cap C_b$, the map $\mu \mapsto \E_\mu[\phi]$ remains continuous, and the minimax theorem still applies.
Restricting to continuous tests raises the question of whether the representation in \eqref{eq:TV-dist-formula-proof-section} for the total-variation distance remains valid with $\Phi_c$ in place of $\Phi$. The following lemma confirms that in the metric space setting of Theorem~\ref{th:weakly-compact} this is indeed so.

\begin{lemma} \label{lem:TV-Cb}
Assume $\Omega$ is a metric space with its Borel $\sigma$-algebra. Then for any $\mu, \nu \in \Mcal_1$ we have $\dist_\TV(\mu,\nu) = \sup_{\phi \in \Phi_c} (\E_\mu[\phi] - \E_\nu[\phi])$.
\end{lemma}

\begin{proof}
The Hahn decomposition of the signed measure $\eta = \mu - \nu$ yields a measurable set $A$ such that the positive part $\eta^+$ is concentrated on $A$ and the negative part $\eta^-$ is concentrated on $A^c$.
The total-variation distance is then $\dist_\TV(\mu,\nu) = \mu(A) - \nu(A) = \eta^+(A)$, the total mass of $\eta^+$.
Fix any closed sets $F_1 \subset A$ and $F_0 \subset A^c$. 
Urysohn's lemma, which applies in any metric space, yields a continuous function $\phi$ with $\phi = 1$ on $F_1$, $\phi = 0$ on $F_0$, and $0 \le \phi \le 1$.
Thus
\[
\E_\eta[\phi] \ge \eta^+(F_1) - \eta^-(A^c \setminus F_0) = \eta^+(F_1) + \eta^-(F_0) - \eta^-(A^c).
\]
We conclude that
\[
\sup_{\phi \in \Phi_c} (\E_\mu[\phi] - \E_\nu[\phi]) \ge \eta^+(F_1) + \eta^-(F_0) - \eta^-(A^c)
\]
for all closed sets $F_1 \subset A$ and $F_0 \subset A^c$.
Taking the supremum of the right-hand side over all such $F_1$ and $F_0$ yields $\eta^+(A)$, since positive finite measures on a metric space with the Borel $\sigma$-algebra are regular; see \citet[Theorem~1.1]{MR1700749}.
We have shown that $\sup_{\phi \in \Phi_c} (\E_\mu[\phi] - \E_\nu[\phi]) \ge \dist_\TV(\mu,\nu)$, and the reverse inequality follows from the fact that $\E_\mu[\phi] - \E_\nu[\phi] \le \dist_\TV(\mu,\nu)$ for all $\phi \in \Phi$.
\end{proof}

We now proceed with the proof of Theorem~\ref{th:weakly-compact}.
Although the main calculation is very similar to the one in the proof of Theorem~\ref{T_kraft}, we emphasize that  it is not subsumed by it. Furthermore, Theorem~\ref{T_kraft} is used in the proof, so the two results are not independent.
Recall that $\Phi_c = \Phi \cap C_b$ denotes the set of continuous tests.

\begin{proof}[Proof of Theorem~\ref{th:weakly-compact}]
Using first Fan's minimax theorem, then an algebraic manipulation and the definition of $R(\phi, \Pcal, \Qcal)$, and finally relaxing the infimum to a larger set, we obtain
\begin{align*}
\inf_{\substack{\mu \in \Pcal \\ \nu \in \Qcal}} \sup_{\phi \in \Phi_c} \big( \E_\nu[\phi] - \E_\mu[\phi] \big)
&= \sup_{\phi \in \Phi_c} \inf_{\substack{\mu \in \Pcal \\ \nu \in \Qcal}} \big( \E_\nu[\phi] - \E_\mu[\phi] \big) \\
&= 1 - \inf_{\phi \in \Phi_c} R(\phi, \Pcal, \Qcal) \\
&\le 1 - R(\Pcal, \Qcal).
\end{align*}
The infimum on the left-hand side is attained and, in view of Lemma~\ref{lem:TV-Cb}, is equal to $\dist_\TV(\Pcal, \Qcal)$.
By Theorem~\ref{T_kraft}, the right-hand side equals 
$\dist_\TV(\cconv(\Pcal), \cconv(\Qcal))$, which is bounded above by $\dist_\TV(\Pcal, \Qcal)$, the distance between two smaller sets.
We have thus shown that
\begin{align*}
\dist_\TV(\Pcal, \Qcal) &= 1 - \inf_{\phi \in \Phi_c} R(\phi, \Pcal, \Qcal) \\
&\le 1-R(\Pcal,\Qcal) \\
&\le \dist_\TV(\cconv(\Pcal), \cconv(\Qcal)) \le \dist_\TV(\Pcal, \Qcal).
\end{align*}
Thus there is equality throughout,  completing the proof.
\end{proof}

We end with the proof of Proposition~\ref{prop:ref-meas-tv-dist}.
We could have used the minimax theorem here too, but this does not seem to simplify matters compared to the proof below, which relies on the Hahn--Banach theorem instead.

\begin{proof}[Proof of Proposition~\ref{prop:ref-meas-tv-dist}]
We regard $\Pcal$ and $\Qcal$ as subsets of $L^1 = L^1(\gamma)$ by identifying measures $\mu$ with their $\gamma$-densities $f_\mu = d\mu/d\gamma$. The total-variation norm is then the $L^1$ norm of the density, $\|\mu\|_\textnormal{TV} = \|f_\mu\|_{L^1}$.
Set $r = \dist_\textnormal{TV}(\conv(\Pcal), \conv(\Qcal)) = \inf \frac12 \|f_\mu - f_\nu\|_{L^1}$, where the infimum extends over all $\mu \in \conv(\Pcal)$ and $\nu \in \conv(\Qcal)$.
Suppose $r > 0$ and note that the convex set $\conv(\Pcal) - \conv(\Qcal)$ is disjoint from $B_{2r}$, the open ball of radius $2r$ in $L^1$.
The Hahn--Banach theorem (see, e.g., \citet[Theorem~II.9.1]{sha_wol_99}) then yields a unit-norm element $\psi \in L^\infty(\gamma)$ such that $\E_{\gamma}[\psi(f_\mu - f_\nu)] \ge 2r$ for all $\mu \in \conv(\Pcal)$, $\nu \in \conv(\Qcal)$. 
We fix a version of $\psi$ such that $|\psi| \le 1$, and define the test $\phi = (1+\psi)/2 \in \Phi$.
It follows that $\E_\mu[\phi] - \E_\nu[\phi] \ge r$ for all $\mu, \nu$ in $\conv(\Pcal)$ and $\conv(\Qcal)$, and then in $\cconv(\Pcal)$ and $\cconv(\Qcal)$ by weak-$*$ continuity. 
We conclude that $\dist_\textnormal{TV}(\cconv(\Pcal), \cconv(\Qcal)) \ge r$. (If $r=0$ this holds trivially.) On the other hand, the reverse inequality holds because $r$ is the distance between the two smaller sets $\conv(\Pcal)$ and $\conv(\Qcal)$.

It remains to confirm that the minimax risk is always attained in the dominated setting.
To do so, observe that for any $\mu \ll \gamma$, the map $\phi \mapsto \E_\mu[\phi] = \E_{\gamma}[f_\mu \phi]$ is continuous in the weak-$*$ topology of $L^\infty(\gamma)$.
Thus the risk $R(\phi, \Pcal, \Qcal)$ is lower semicontinuous in this topology. Since the set of all ($\gamma$-equivalence classes of) tests is weak-$*$ compact in $L^\infty(\gamma)$ thanks to the Banach--Alaoglu theorem, there exists one that minimizes the risk.
\end{proof}

\section{The case of a singleton \texorpdfstring{$\Qcal$}{Q}}\label{sec:singleton}

The case of a singleton $\Qcal =\{\nu\}$ is of particular interest in relation to the recent work of~\cite{larsson2024numeraire}. As we formalize below, a corollary of that work (not explicitly stated there, but stated and proved below) is that a nontrivial test exists if and only if $\nu$ does not lie in the effective null hypothesis $\Peff$, defined below. A trivial corollary of our main theorem in this paper is that a nontrivial test exists if and only if $\nu$ does not lie in $\cconv(\Pcal)$. What exactly is the relationship between these two geometric objects appearing in the above statements? This section provides the answer. 

We begin with some definitions. Given a set of distributions $\Pcal$, an \emph{e-variable} is a $[0,\infty]$-valued random variable whose expectation is at most one under every $\mu \in \Pcal$. The set of all e-variables for $\Pcal$ is given by the \emph{polar} set of $\Pcal$,
\[
\Ecal = \{ Z \geq 0: \E_\mu[Z] \leq 1 \text{ for all } \mu \in \Pcal \}.
\]
\cite{larsson2024numeraire} then define the polar of $\Ecal$ (the \emph{bipolar} of $\Pcal$) as
\[
{\color{black}\Peff = \{ \mu \in \Mcal_+: \E_\mu[Z] \leq 1 \text{ for all } Z \in \Ecal \},}
\]
and call it the \emph{effective null hypothesis}. 
Note that every element of $\Peff$ has mass at most one; just consider the e-variable $Z = 1$.
No test can have nontrivial power against any distribution in $\Peff$. Indeed, as a corollary of their main theorem, we have the following.

\begin{corollary}
$\nu \in \Peff$ if and only if for any test $\phi$, $\E_\nu[\phi] \leq \sup_{\mu \in \Pcal} \E_\mu[\phi]$.
\end{corollary}
\begin{proof}
    For the forward direction, assume $\nu\in\Peff$, and consider any test $\phi$. If the constant $c = \sup_{\mu \in \Pcal} \E_\mu[\phi]$ is strictly positive, then $\phi/c$ is an e-variable, and $\nu \in \Peff$ implies that $\E_\nu[\phi/c] \leq 1$.  If instead $c=0$, then $\phi+1$ is an e-variable and $\nu \in \Peff$ implies $\E_\nu[\phi] = 0$. 
    
    For the reverse direction, assume $\nu \notin \Peff$. Then there exists an e-variable $Z$ such that $\E_\nu[Z] > 1$. By monotone convergence we may assume that $Z \leq n$ for some large $n \in \N$. Then $\phi = Z/n$ is a test, and we have $\E_\nu[\phi] > 1/n \geq \sup_{\mu \in \Pcal} \E_\mu[\phi]$, a contradiction.
\end{proof}

On the other hand, we also have the following immediate corollary of Theorem~\ref{T_kraft}.
\begin{corollary}
$\nu \in \cconv(\Pcal)$ if and only if for any test $\phi$, $\E_\nu[\phi] \leq \sup_{\mu \in \Pcal} \E_\mu[\phi]$. 
\end{corollary}

Note that $\cconv(\Pcal) \neq \Peff$. For example, the zero measure is always in $\Peff$ but never in $\cconv(\Pcal)$. However, it is also not the case that $\Peff$ contains $\cconv(\Pcal)$: the latter may contain finitely additive elements that are absent from the former. {\color{black}Nevertheless, the above corollaries immediately imply the following theorem.}

\begin{theorem}\label{th:equal-point-null}
    $\cconv(\Pcal) \cap \Mcal_1 = \Peff \cap \Mcal_1$.
\end{theorem}

The above theorem explains why, with their focus on a singleton $\Qcal = \{\nu\}$, \citet{larsson2024numeraire} managed to avoid $\ba$: the qualitative behavior of testability (i.e.\ the existence of a nontrivial test) does not require $\ba$. Further, the quantitative statements of~\cite{larsson2024numeraire} in terms of infimum KL divergence from $\nu$ to $\Pcal$ also avoided $\ba$. However, even for a singleton $\Qcal = \{\nu\}$, the TV distances in Theorem~\ref{T_kraft} could not have used $\Peff$ instead of $\cconv(\Pcal)$; we provide an example soon. 

{\color{black}The precise relationship between $\Peff$ and $\cconv(\Pcal)$ is however more subtle. It is characterized in the following theorem, whose proof is presented in Appendix~\ref{S_proof_T_Peff}. This result also immediately implies Theorem~\ref{th:equal-point-null}.
In the following statement, \emph{maximal} is with respect to the setwise ordering, and the \emph{solid hull} of a set $\Acal \subset \Mcal_+$ is the set $\{\mu \in \Mcal_+ \colon \mu \le \nu \text{ for some } \nu \in \Acal\}$.} 

\begin{theorem} \label{T_Peff}
Every $\mu \in \Peff$ is setwise dominated by some maximal element $\mu' \in \Peff$. The maximal element $\mu'$ is the countably additive part of some $\mu'' \in \cconv(\Pcal)$. As a consequence, $\Peff$ is the solid hull of the countably additive parts of $\cconv(\Pcal)$.
\end{theorem}

We continue with a simple yet important example, which demonstrates that even in the case of a simple alternative
$\Qcal=\{\nu\}$, in Theorem~\ref{T_kraft} it does not suffice to only consider $\cconv(\Pcal) \cap \Mcal_1$ (which restricts attention to countably additive probability measures, such as $\nu$), but we indeed have to consider finitely additive measures in $\cconv(\Pcal)$.
This example has the subtle feature that it \emph{assumes} that 
\begin{equation} \label{eq:no-diff-pow}
\text{there is no diffuse probability measure on $\textnormal{Pow}(\R)$},
\end{equation} 
where ``$\textnormal{Pow}$'' denotes the power set (the set of all subsets).
This assumption is discussed further in Remark~\ref{rem:c-hyp} below; in particular, it is implied by the continuum hypothesis. The upshot of that discussion is that the standard setting of probability theory allows for \eqref{eq:no-diff-pow}. Without further undesirable assumptions, situations like the one in the example below can occur and must be covered by the general theory.
It is an interesting open question whether similar examples can be constructed without any assumptions such as \eqref{eq:no-diff-pow}.

\begin{example}  \label{ex:251212}
Let $\Omega = [-1,1]$ and define $$\Fcal = \{A \cup B \colon A \in \textnormal{Pow}([-1,0]), B \in \textnormal{Bor}([0,1])\}.$$ Here ``$\textnormal{Pow}$'' denotes the power set and ``$\textnormal{Bor}$'' the Borel $\sigma$-algebra.
Then $\Fcal$ is the $\sigma$-algebra generated by the Borel subsets of $[0,1]$ and \emph{all} subsets of $[-1,0]$.
We assume that \eqref{eq:no-diff-pow} holds.
Consider
\[
\Pcal = \left\{\frac12\delta_x + \frac12\delta_{-x} \colon x \in (0,1)\right\}, \quad \Qcal = \{\nu\}, \quad \nu = \mathrm{Uni}[0,1].
\]
We will show that
\begin{equation} \label{eq:ba_example}
\dist_\TV(\cconv(\Pcal) \cap \Mcal_1,\Qcal) = 1 \quad \text{but} \quad
\dist_\TV(\cconv(\Pcal),\Qcal) = \frac12.
\end{equation}
This demonstrates that, in general, one cannot restrict attention to the countably additive elements of $\cconv(\Pcal)$ and $\cconv(\Qcal)$ in Theorem~\ref{T_kraft}.
Put differently, in order for Theorem~\ref{T_kraft} to hold, the closure of $\conv(\Pcal)$ and $\conv(\Qcal)$ cannot be taken in the space $\Mcal_1$ with the topology $\sigma(\Mcal_1,\Lcal)$. Finitely additive measures are really needed.

To argue \eqref{eq:ba_example}, we first observe that every $\mu \in \cconv(\Pcal)$ is symmetric around zero in the sense that $\mu(A \cup (-A)) = 2 \mu(A) = 2 \mu(-A)$ for every $A \in \textnormal{Bor}([0,1])$, where $-A = \{-x \colon x \in A\}$. Indeed, this holds for all $\mu \in \Pcal$, thus for all $\mu \in \conv(\Pcal)$, and finally for all $\mu \in \cconv(\Pcal)$ by weak-$*$ continuity.

Consider now any $\mu \in \cconv(\Pcal) \cap \Mcal_1$.
Its restriction to $[-1,0]$ is a positive countably additive measure on $\textnormal{Pow}([-1,0])$, and is therefore concentrated on a countable set $S_0 \subset [-1,0]$ thanks to Proposition~\ref{P:1}. The symmetry property applied with $A = -S_0$ then shows that $\mu$ is concentrated on $S = S_0 \cup (-S_0)$.
But since any countable set is a $\nu$-nullset, we have $\dist_\TV(\mu,\nu) \ge \mu(S) - \nu(S) = 1$, and this yields the first part of \eqref{eq:ba_example}.

We now focus on the second part of \eqref{eq:ba_example}. The test $\phi = \bm1_{[0,1]}$ certifies that $\dist_\TV(\cconv(\Pcal),\Qcal)$ is at least $\frac12$, and the point is to show that this is achieved for some $\mu \in \cconv(\Pcal)$.
To do this we start with a net $(\eta_\alpha)$ of finitely supported probability measures on $(0,1)$ that converges to $\nu$; this is possible thanks to Proposition~\ref{P:2}.
The symmetrizations $\mu_\alpha(A \cup B) = \frac12(\eta_\alpha(-A) + \eta_\alpha(B))$ for $A \in \textnormal{Pow}([-1,0])$ and $B \in \textnormal{Bor}([0,1])$  then belong to $\conv(\Pcal)$. 
On passing to a convergent subnet, we may assume that $(\mu_\alpha)$ converges to some limit $\mu \in \cconv(\Pcal)$.
The restriction of $\mu$ to $[0,1]$ is equal to $\nu/2$ by construction, and it follows that $\dist_\TV(\mu, \nu) = \frac12$.
This completes the proof of the second part of \eqref{eq:ba_example}.

It is worth noting that the restriction of $\mu$ to $[-1,0]$ must be purely finitely additive.
Indeed, its countably additive part is atomless (due to the symmetry property and the fact that on $[0,1]$, $\mu$ is proportional to the uniform  distribution) and therefore zero due to Proposition~\ref{P:1}.
\end{example}

\begin{remark} \label{rem:c-hyp}
The assumption \eqref{eq:no-diff-pow} is known to be consistent with the Zermelo--Fraenkel set theory with the axiom of choice (ZFC). 
One may thus add \eqref{eq:no-diff-pow} as an axiom without rendering any theorems proved in ZFC false.
Indeed, there are important models of ZFC, such as G\"odel's constructible universe $L$, where \eqref{eq:no-diff-pow} can be proved as a theorem. Specifically, \cite{MR2514} showed that the continuum hypothesis holds in $L$, and by a theorem of
\citet[Th\'eor\`eme~I]{banach_kuratowski_1929}, this implies \eqref{eq:no-diff-pow}. See also Theorem~13.20 and Theorem~10.1 in \citet{MR1940513} for a textbook treatment.
For these reasons, it seems sensible that our general theory should not depend on any axioms or hypotheses that disallow \eqref{eq:no-diff-pow}.
\end{remark}

We end by discussing an important implication of our theorem for the existence of uniformly powered bounded e-variables for composite $\Qcal$. An e-variable $Z$ for $\Pcal$ is said to be uniformly powered against $\Qcal$ if $\inf_{\nu \in \Qcal} \E_\nu[Z] > 1$. 
\begin{corollary}
    Let $\Pcal, \Qcal \subset \Mcal_1$. Then there exists a bounded e-variable for $\Pcal$ that is uniformly powered against $\Qcal$ if and only if $\dist_\TV(\cconv(\Pcal), \cconv(\Qcal)) > 0$.
\end{corollary}
\begin{proof}
    In the forward direction, let $Z$ be the uniformly powered e-variable that lies in $[0,B]$ for some finite $B$. Then clearly, $Z/B$ is a nontrivial test with type I error at most $1/B$ and worst-case power strictly greater than $1/B$, which implies that $\dist_\TV(\cconv(\Pcal), \cconv(\Qcal)) > 0$ by Theorem~\ref{T_kraft}, completing the first direction.

    Assume now that $\dist_\TV(\cconv(\Pcal), \cconv(\Qcal)) > 0$. By Theorem~\ref{T_kraft}, there exists 
    a test $\phi$ such that
    $\inf_{\nu \in \Qcal} \E_\nu[\phi] >  \sup_{\mu \in \Pcal} \E_\mu[\phi]$.  If $\sup_{\mu \in \Pcal} \E_\mu[\phi] = 0$, set $Z = 1 + \phi$. Otherwise, set $Z = \phi / \sup_{\mu \in \Pcal} \E_\mu[\phi]$. Then $Z$ is a bounded e-variable that is uniformly powered against $\Qcal$, completing the proof.
\end{proof}

Following \citet[Chapter 3]{ramdas2024hypothesis},  
we say that an e-variable $Z$ is uniformly \emph{e-powered} against $\Qcal$ if $\inf_{\nu \in \Qcal} \E_\nu[\log Z] > 0$.
Theorem 3.14 in the above book then proves that a uniformly powered bounded e-variable exists if and only if a uniformly e-powered bounded e-variable exists, and so the above theorem could have also been stated with ``powered'' being replaced by ``e-powered''. This extends some results in the above book, and in~\cite{zhang2024existence}.

\subsection*{Acknowledgments}

We thank Sivaraman Balakrishnan for useful discussions regarding Le~Cam's work.
ML acknowledges support from the National Science Foundation under grant
NSF DMS-2510965.
AR acknowledges support from the Sloan Fellowship and the National Science Foundation under grant NSF DMS-2310718.

\bibliography{bibliography}
\bibliographystyle{plainnat}

\appendix

\section{Further comments on the choice of topology}\label{app:topology}

In the introduction we showed why the TV topology and the weak convergence topology cannot, in general, be used to obtain characterizations of testability; see Examples~\ref{eg:van-de-waart} and \ref{ex:intro-2} and the adjacent discussion.
An alternative approach would be to look for topologies on the set $\Phi$ of all tests, making this set compact. This could then be leveraged to apply the minimax theorem without compactness of the closed convex hulls of $\Pcal$ and $\Qcal$.
This idea works in the dominated setting of Theorem~\ref{T:classical}, and Le Cam's approach to embed $\Phi$ into the larger, compact, space of generalized tests can be viewed as an extension of this idea to the general case, with important drawbacks as discussed in Appendix~\ref{sec:gen-tests}.

{\color{black}
It is then natural to look for topologies that make $\Phi$ itself compact in the general non-dominated setting.
A natural attempt is to consider the space $[0,1]^\Omega$ of all $[0,1]$-valued functions on $\Omega$, equipped with the \emph{topology of pointwise convergence}.
This is simply the product topology on $[0,1]^\Omega$, which is compact by Tychonoff's theorem.
There are however two issues with this approach. First, $\Phi \subset [0,1]^\Omega$ is the proper subspace consisting of \emph{measurable} functions, which is \emph{not} compact in general.
Moreover, the linear functions $\phi \mapsto \E_\mu[\phi]$ fail to be continuous (or semicontinuous) in general.
Thus, even if one could meaningfully extend the definition of the risk $R(\phi)$ to \emph{every} $\phi$ in the compact space $[0,1]^\Omega$, the semicontinuity hypothesis of the minimax theorem would not be satisfied, and the theorem could not be used.
}

To restore (semi-)continuity, one has to work with a stronger topology.
A very strong topology is the one induced by the supremum norm $\|\phi\|_\infty = \sup_{\omega \in \Omega} |\phi(\omega)|$. The above linear maps $\phi \mapsto \E_\mu[\phi]$ are now continuous, even Lipschitz, thanks to the inequality $|\E_\mu[\phi]-\E_\mu[\psi]| \leq \|\mu\|_\TV\|\phi-\psi\|_\infty$.
{\color{black}However, compactness of both $\Phi$ and $[0,1]^\Omega$ now fails, again rendering the minimax theorem inapplicable.}
Indeed, given any sequence $\omega_n$ of mutually distinct points, the tests $\phi_n = \bm1_{\{\omega_n\}}$ are at a distance $\|\phi_n - \phi_m\|_\infty = 1$ apart, and thus form a sequence with no convergent subsequence.

At this point it may be tempting to consider the topology $\sigma(\Phi,\Mcal)$, which is the \emph{weakest} topology on $\Phi$ that makes the linear maps $\phi \mapsto \E_\mu[\phi]$ continuous for all $\mu \in \Mcal$.
Unfortunately, compactness of $\Phi$ still fails, and the net $\phi_A = \bm1_A$ indexed by finite subsets $A \subset \Omega = [0,1]$ again furnishes an example.
Indeed, suppose for contradiction that compactness holds. Then there is a convergent subnet $\phi_B$ with some limit $\phi$. This limit is the constant function $\phi = 1$. To see this, note that for each $\omega \in [0,1]$, $\phi_B(\omega)=1$ eventually, so $\phi(\omega) = \E_{\delta_\omega}[\phi] = \lim_B \E_{\delta_\omega}[\phi_B] = 1$, where we used that $\lim_B \E_\mu[\phi_B] = \E_\mu[\phi]$ for every $\mu \in \Mcal$ by definition of the topology.
However, if we now take $\mu$ to be Lebesgue measure, we obtain the contradiction $0 = \lim_B \E_\mu[\phi_B] = \E_\mu[\phi] = 1$.
This shows that $\phi_A$ has no convergent subnet, and hence that $\Phi$ is not compact.

One thus arrives at the conclusion that either $\Phi$ has to be ``compactified'' (for example by enlarging it to generalized tests as  proposed by Le Cam, with its drawbacks discussed in Appendix~\ref{sec:gen-tests}), or one has to find a way to compactify the space of measures rather than tests. The latter is the route taken in the present paper (by enlarging the space of measures), which leads us naturally --- even necessarily --- to working with finitely additive measures.

We hope that this discussion gives a further perspective on the delicate interplay between the choice of ambient space, its topology, and how these must be tailored to fit the confines of the minimax theorem. Unfortunately, we are aware of several published works which incorrectly use minimax theorems to achieve important conclusions.  For example~\cite{gul2017minimax} mistakenly invoke Sion's minimax theorem after using the product topology on $[0,1]^\Omega$, which as we discussed above destroys continuity. 
As another example~\citet[Chapter 6]{levy2008principles} invokes von Neumann's minimax theorem after endowing $\Phi$ with  $\|\cdot\|_\infty$, which as we discussed above renders $\Phi$ non-compact.
While we know of several more examples, it is not our goal to point out all such errors in other works, but merely to point to the subtleties involved with the careful choices necessary for everything to work out. It would also be of interest to see whether the corresponding conclusions in these works are incorrect, or if their proofs can be fixed using techniques developed here.

\section{Generalized tests} \label{sec:gen-tests}

In this section we elaborate on the distinction between Theorem~\ref{T_kraft} and Le~Cam's result  on generalized tests, Theorem~\ref{th:gen-test}.
After all, the two results lead to different conclusions regarding which hypotheses are testable.
While this is not a contradiction mathematically (the two results make use of \emph{tests} and \emph{generalized tests}, respectively), it seems worthwhile to look closer at how they differ in order to understand the implications.

It is useful to return to the basic example in the introduction, where $\Omega = [0,1]$, $\Fcal$ is the Borel $\sigma$-algebra, $\Pcal = \{\delta_x\colon x \in [0,1]\}$, and $\Qcal = \{\nu\}$ with $\nu = \mathrm{Uni}[0,1]$, the uniform distribution on $[0,1]$.
The weak-$*$ convex closure $\cconv(\Pcal)$ contains \emph{all} probability measures on $[0,1]$, in particular $\nu$, so Theorem~\ref{T_kraft} implies that $R(\Pcal,\Qcal) = 1$. Thus there is no nontrivial test for this pair of hypotheses.
On the other hand, we have $\dist_\textnormal{TV}(\conv(\Pcal),\nu) = 1$, so Le~Cam's Theorem~\ref{th:gen-test} implies that there is a generalized test $\phi^*$ that separates $\Pcal$ and $\Qcal$.

In fact, it is not hard to construct such a generalized test. For any $\mu \in \Mcal$, let $\mu^a$ be its absolutely continuous part with respect to $\nu$. The map $\mu \mapsto \mu^a$ is linear and continuous (even a contraction) in the total-variation norm. 
Now define $\phi^*(\mu) = \mu^a(\Omega)$, the total mass of $\mu^a$. Then $\phi^*$ is a continuous linear functional on $\Mcal$, and $\phi^*(\mu) \in [0,1]$ for all $\mu \in \Mcal_1$. That is, $\phi^*$ is a generalized test.
Furthermore, $\phi^*(\nu) = 1$ and $\phi^*(\mu) = 0$ for all $\mu \in \Pcal$.

The key issue here is that $\phi^*(\mu)$ may depend on the entire \emph{distribution} $\mu$, not just on the \emph{sample} $\omega \in [0,1]$ that is produced by $\mu$.
In this example, to compute the value of the generalized test it is not enough to know the realized sample $\omega$---in fact, its value is irrelevant---but rather whether it came from an absolutely continuous distribution or not.
This can be thought of as \emph{additional information} that the experimenter would need to acquire in addition to the sample itself.

More generally, Theorem~\ref{T_kraft} shows that whenever no nontrivial test exists, but a nontrivial generalized test does exist, some form of additional information is \emph{always} necessary in order to evaluate the generalized test.

Interestingly, as Le~Cam points out in the remark on page~476 of \citet{le2012asymptotic}, although generalized tests need not be representable by a measurable function $\phi$ on the original sample space $\Omega$ (i.e., by a test), it is always possible to extend $(\Omega, \Fcal)$ to a larger sample space $(\widetilde\Omega, \widetilde\Fcal)$ in such a way that every generalized test \emph{can} be represented by a measurable function on $\widetilde\Omega$.
As such, this extension furnishes the additional data needed to evaluate generalized tests.
However, since in applications the original space $(\Omega,\Fcal)$ typically models the information available to the experimenter, passing to the extension $(\widetilde\Omega, \widetilde\Fcal)$ may alter the information structure in an unintended way. The example above illustrates this quite clearly: here the sample $\omega$ needs to be accompanied with a ``tag'' indicating whether or not it was drawn from an absolutely continuous distribution.

Lastly, it may be of interest to know something about the structure of the extended sample space $(\widetilde\Omega, \widetilde\Fcal)$.
The construction outlined in Chapter~1.6 of \citet{le2012asymptotic} is very general but also quite abstract and difficult to interpret.
In the setting of the present paper, there is a significantly less abstract construction which we now describe.

We start with the original measurable space $(\Omega,\Fcal)$ and its space $\Mcal$ of finite signed measures.
Zorn's lemma implies that there exists a set $\{\mu_i \colon i \in I\}$ of pairwise singular probability measures that is maximal in the collection of all such sets, ordered by inclusion.
Every $\mu \in \Mcal$ is then absolutely continuous with respect to some countable combination of the $\mu_i$.
Note that the index set $I$ need not be countable. {\color{black}(If $I$ is countable there is a single dominating measure $\gamma$, and $\Mcal$ can be identified with $L^1(\gamma)$.)}
No extension of $\Omega$ is needed in this case, so the following  is only relevant if $I$ is uncountable.

For each $i \in I$ we let $(\Omega_i,\Fcal_i)$ be a copy of $(\Omega, \Fcal)$ and consider the disjoint union
\[
\widetilde\Omega = \bigsqcup_{i \in I} \Omega_i,
\]
equipped with the $\sigma$-algebra $\widetilde\Fcal$ consisting of all sets $A \subset \widetilde\Omega$ such that $A \cap \Omega_i \in \Fcal_i$ for all $i \in I$. 
We further let $\widetilde\mu$ be the measure on $\widetilde\Omega$ that restricts to $\mu_i$ on each $\Omega_i$,
\[
\widetilde\mu(A) = \sum_{i \in I} \mu_i(A \cap \Omega_i), \quad A \in \widetilde\Fcal.
\]
If $I$ is uncountable, which is the relevant case, the measure $\widetilde\mu$ is not $\sigma$-finite.
Nonetheless, $\widetilde\Fcal$-measurable functions $f$ can be integrated provided they are $\widetilde\mu$-integrable, meaning that the restriction $f_i = f|_{\Omega_i}$ belongs to $L^1(\mu_i)$ for each $i \in I$ and the norm $\|f\|_{L^1(\widetilde\mu)} = \sum_{i \in I} \|f_i\|_{L^1(\mu_i)}$ is finite.
Note that this requires $f_i = 0$ for all but countably many $i \in I$.
The space $L^1(\widetilde\mu)$ of (equivalence classes of) $\widetilde\mu$-integrable functions with the norm $\|\cdot\|_{L^1(\widetilde\mu)}$ is a Banach space, known as the $\ell^1$-direct sum of $L^1(\mu_i)$, $i \in I$.

The key point is that $L^1(\widetilde\mu)$ is isometrically isomorphic to the Banach space $(\Mcal,\|\cdot\|_\textnormal{TV})$.
Indeed, 
any $\mu \in \Mcal$ can be mapped to the unique $f \in L^1(\widetilde\mu)$ whose restriction $f_i$ to $\Omega_i$ is the density $d\mu^a/d\mu_i$, where $\mu^a$ is the absolutely continuous part of $\mu$ with respect to $\mu_i$. Note that $\mu^a$ depends on $i$, although this is not indicated in the notation, and all but countably many $f_i$ are zero.
The mapping $\mu \mapsto f$ is furthermore linear and surjective.
Thanks to the maximality property of the family $\{\mu_i \colon i \in I\}$, it preserves the norm $\|\mu\|_\textnormal{TV} = \|f\|_{L^1(\widetilde\mu)}$, and is thus an isometry.

In this way the elements of $\Mcal$ can be regarded as $\widetilde\mu$-absolutely continuous measures on the (potentially much larger) sample space $\widetilde\Omega$.
Furthermore, the dual space $\Mcal'$ can be identified with the dual of $L^1(\widetilde\mu)$, which is $L^\infty(\widetilde\mu)$. 
Thus any generalized test can be represented as a measurable function on $\widetilde\Omega$, which was the goal of the construction.
(The $L^1$-$L^\infty$ duality holds in the present non-$\sigma$-finite case because $\widetilde\mu$ is \emph{(strictly) localizable}; see \citet[Chapters~21 and~24]{MR2462280}, specifically Definition~211E, Theorem~211L(d), and Theorem~243G.)

Let us clarify how this construction encodes the ``additional information'' mentioned previously.
The observed sample $\widetilde\omega \in \widetilde\Omega$ belongs to $\Omega_i$ for exactly one $i$, so it must have been generated by a distribution $\mu$ with nonzero density relative to $\mu_i$. 
Thus $i$, which can be inferred from $\widetilde\omega$, is the ``tag'' that conveys information about what might be called the ``absolute continuity type'' of the generating distribution $\mu$.
Since the spaces $\Omega_i$ are tagged, but otherwise identical, copies of $\Omega$, the sample $\widetilde\omega$ can indeed be regarded as a point of $\Omega$, supplemented with the tag $i$.
The generalized test, regarded as a bounded measurable function on $\widetilde\Omega$, can now be evaluated as long as this supplemented sample is observed.
Furthermore, ordinary (not generalized) tests correspond to functions on $\widetilde\Omega$ that are the same on every $\Omega_i$, that is, do not depend on the tag $i$. 

Lastly, we note that the considerable complexity of the space $\widetilde\Omega$ can sometimes be reduced. This happens when the family $\{\mu_i \colon i \in I\}$ can be replaced by a smaller family that still dominates $\Pcal$ and $\Qcal$ but perhaps not \emph{every} measure in $\Mcal$.
In particular, in the dominated case of Theorem~\ref{T:classical}, the reference measure itself can be used as the single member of the family, and one can take $\widetilde\Omega = \Omega$ as noted previously.

\begin{remark}
Le Cam has shown that the minimax risk, {\color{black}computed over all generalized tests}, can always be attained by a generalized test.
The extended space $\widetilde\Omega$ can be used to construct a net of (non-generalized) tests that converge in the weak-$*$ topology of $L^\infty(\widetilde\mu)$ to a given generalized test.
Specifically, let $\phi^* \in L^\infty(\widetilde\mu)$ be a generalized test, and let $\phi^*_i$ be its restriction to $\Omega_i$ for each $i \in I$.
For any countable subset of indices $J \subset I$ one can find mutually disjoint sets $A_i \in \Fcal$, $i \in J$, such that $\mu_i(A_i) = 1$ for all $i \in J$, where we regard $A_i$ as a subset of $\Omega_i$.
Now define $\phi_J = \sum_{i \in J} \phi^*_i \bm1_{A_i}$, regarded as a function on $\Omega$.
Then $\phi_J$ is a (non-generalized) test and satisfies $\phi_J = \phi^*_i$, $\mu_i$-a.s., for all $i \in J$.
Furthermore, for every $\mu \in \Mcal$ we have $\E_\mu[\phi_J] = \phi^*(\mu)$ whenever $J$ contains all $i \in I$ such that $d\mu^a/d\mu_i$ is not zero.
Thus the $\phi_J$ form a net that converges to $\phi^*$ in the weak-$*$ topology of $L^\infty(\widetilde\mu)$.
\end{remark}

\section{Some measure-theoretic results}  \label{A:measure-theoretic}

\begin{proposition} \label{P:2}
Let $(\Omega,\Fcal)$ be a measurable space and let $\Pcal=\{\delta_\omega \colon \omega \in \Omega\}$. Then $\cconv(\Pcal) = \ba_1$. 
\end{proposition}

\begin{remark}
Let us point out that if $\Omega = [0,1]$ and $\Pcal$ consists of all $\mu$ with a Lebesgue density, then $\cconv(\Pcal)$ is \emph{not} equal to all of $\ba_1$. 
This is because every $\mu \in \cconv(\Pcal)$ satisfies $\mu(A) = 0$ for all Lebesgue nullsets $A$.
In particular, point masses $\delta_\omega$ are not weak-$*$ limit points of $\conv(\Pcal)$.
For example, the sequence $\mu_n = \mathrm{Uni}[0,1/n]$ does \emph{not} converge to $\delta_0$. In fact, this sequence does not converge at all. 
This is demonstrated, for instance, by the function $f(\omega) = \sin(\log(1/\omega))$ for $\omega \ne 0$ and $f(0) = 0$, which satisfies $\E_{\mu_n}[f] = \sin({\pi}/{4} + \log(n))/\sqrt{2}$. 
This does not converge as $n \to \infty$, and then neither can $\mu_n$.
However, weak-$*$ compactness implies that $\mu_n$ does have a convergent \emph{subnet}. The limit is a purely finitely additive probability measure $\mu$ that assigns unit mass to any interval $[0,\epsilon]$ with $\epsilon > 0$.
\end{remark}

\begin{proof}[Proof of Proposition~\ref{P:2}]
Fix any $\mu \in \ba_1$. We construct a net in $\conv(\Pcal)$ as follows.
Let $\Pi$ be the set of all partitions $\pi = \{A_1,\ldots,A_n\}$ of $\Omega$ into finitely many nonempty measurable sets, partially ordered by refinement.
For each such $\pi$, choose points $\omega_i \in A_i$, $i=1,\ldots,n$, and define
\begin{equation} \label{eq:mu-pi}
\mu_\pi = \sum_{i=1}^n \mu(A_i) \delta_{\omega_i} \in \conv(\Pcal).    
\end{equation}
Consider now any $B \in \Fcal$. 
For any $\pi \in \Pi$ and any $\pi' \in \Pi$ that is a common refinement of $\pi$ and the partition $\{B, B^c\}$, we have $\mu_{\pi'}(B) = \mu(B)$.
It follows that the net $(\mu_\pi(B))_{\pi \in \Pi}$ in $\R$ converges to $\mu(B)$.
Since this holds for every $B \in \Fcal$, we conclude that the net $(\mu_\pi)_{\pi \in \Pi}$ in $\conv(\Pcal)$ converges to $\mu$.
\end{proof}

Recall that $\textnormal{Pow}(\R)$ refers to the power set of $\R$.

\begin{proposition} \label{P:1}
Assume \eqref{eq:no-diff-pow} and let \(\mu\) be a positive finite measure on \((\R,{\rm Pow}(\R))\). Then \(\mu\) is purely atomic, i.e., the set
\(
S=\{x\in \R:\mu(\{x\})>0\}
\)
is at most countable and for every \(A\subset \R\) we have
\[
\mu(A)=\sum_{x\in A\cap S}\mu(\{x\}).
\]
\end{proposition}

\begin{proof}
For each $k \in \N$, the set $S_k = \{x \in \R \colon \mu(\{x\}) > k^{-1}\}$ must be finite. Thus the set $S = \bigcup_{k \in \N} S_k$ is at most countable. Now set $\mu_d(A) = \mu(A) - \mu(A \cap S)$ for all $A \subset \R$. This defines a positive, finite, and diffuse measure on $\textnormal{Pow}(\R)$. By \eqref{eq:no-diff-pow}, $\mu_d = 0$, and this completes the proof.
\end{proof}

\section{Proof of Theorem~\ref{T_Peff}} \label{S_proof_T_Peff}

The first part of Theorem~\ref{T_Peff} states that every element of $\Peff$ is setwise dominated by some maximal element of $\Peff$. This is shown in Lemma~\ref{L_Peff_maximal_exists} below, building on several auxiliary results which we now develop.
{\color{black} We say that a net $\mu_\alpha$ in $\ba_+$ is \emph{increasing} if $\mu_\alpha \ge \mu_\beta$ setwise whenever $\alpha \ge \beta$. It is bounded if the total masses $\mu_\alpha(\Omega)$ have a finite upper bound.}

\begin{lemma} \label{L_incr_net_ca}
Let $\mu_\alpha$ be a bounded increasing net in $\ba_+$ and define $\mu(A) = \sup_\alpha \mu_\alpha(A)$, $A \in \Fcal$. {\color{black}Then $\mu$ belongs to $\ba_+$ and} $\mu_\alpha$ converges to $\mu$ in the weak-$*$ topology. If each $\mu_\alpha$ is countably additive, then so is $\mu$.
\end{lemma}

\begin{proof}
The boundedness assumption means that $\sup_\alpha \mu_\alpha(\Omega) < \infty$, so $\mu$ is well-defined. Moreover, a bounded increasing net in $\R$ converges to its supremum, so we have $\mu(A) = \lim_\alpha \mu_\alpha(A)$ for all $A \in \Fcal$. For any disjoint $A, B \in \Fcal$ we pass to the limit in the identity $\mu_\alpha(A \cup B) = \mu_\alpha(A) + \mu_\alpha(B)$ to see that $\mu \in \ba_+$. For bounded nets in $\ba$, setwise convergence is equivalent to weak-$*$ convergence, so $\mu_\alpha \to \mu$ in the weak-$*$ topology.
Lastly, suppose each $\mu_\alpha$ is countably additive, and consider an arbitrary countable family of mutually disjoint sets $A_i \in \Fcal$. Monotonicity and finite additivity of $\mu$ yields
\[
\mu\left( \bigcup_{i=1}^\infty A_i \right) \ge \mu\left( \bigcup_{i=1}^n A_i \right) = \sum_{i=1}^n \mu( A_i )
\]
for all $n$, while countable additivity of $\mu_\alpha$ and the fact that $\mu$ dominates $\mu_\alpha$ yields
\[
\mu_\alpha\left( \bigcup_{i=1}^\infty A_i \right) = \sum_{i=1}^\infty \mu_\alpha(A_i) \le \sum_{i=1}^\infty \mu(A_i)
\]
for all $\alpha$. Passing to the limit in $n$ and $\alpha$, respectively, shows that $\mu$ is countably additive.
\end{proof}

\begin{lemma} \label{L_incr_net_Peff}
Let $\mu_\alpha$ be an increasing net in $\Peff$. Then $\mu_\alpha$ converges in the weak-$*$ topology to an element $\mu$ of $\Peff$.
\end{lemma}

\begin{proof}
By Lemma~\ref{L_incr_net_ca}, $\mu_\alpha$ converges in the weak-$*$ topology to some nonnegative countably additive measure $\mu$. This measure belongs to $\Peff$ because the property $\int f d\mu_\alpha \le 1$ for all bounded $f \in \Ecal$ is preserved under weak-$*$ convergence.
{\color{black}Monotone convergence then yields $\int f d\mu \le 1$ for \emph{all} $f \in \Ecal$.}
\end{proof}

\begin{lemma} \label{L_Peff_maximal_exists}
Every $\mu_0 \in \Peff$ is dominated by some maximal element of $\Peff$.
\end{lemma}

\begin{proof}
{\color{black}
Let $D = \{\nu \in \Peff \colon \nu \ge \mu_0\}$ be the set of elements of $\Peff$ that dominate $\mu_0$ setwise.
Every chain in $D$ is an increasing net, which thanks to Lemma~\ref{L_incr_net_Peff} admits an upper bound in $\Peff$, hence in $D$.
Zorn's lemma then yields a maximal element of $D$, which must also be maximal in $\Peff$.
}
\end{proof}

The remaining parts of Theorem~\ref{T_Peff}  require the introduction of a third object: one must view the null hypothesis $\Pcal$ as a subset of $\ba$ and consider an appropriate extension of the effective null, namely,
\begin{equation} \label{eq_Peff_ba}
\Peff^\ba = \{\mu \in \ba_+ \colon \int f d\mu \le 1 \text{ for all } f \in \Ecal_b\},
\end{equation}
where $\Ecal_b$ is the set of all bounded e-variables.
Note that $\Peff$ consists precisely of the countably additive elements of $\Peff^\ba$.

\begin{lemma} \label{L_Peff_ba_rep}
$\Peff^\ba = \ba_+ \cap \cconv(\Pcal - \ba_+)$.
\end{lemma}

\begin{proof}
The polars and bipolars below are understood with respect to the dual pair $(\Lcal, \ba)$, and we follow the notation and definitions of \cite{sha_wol_99}.
We first show that
\begin{equation} \label{eq_Eb_ba_rep}
\Ecal_b = (\Pcal - \ba_+)^\circ.
\end{equation}
The right-hand side equals $\{f \in \Lcal \colon \int f d(\mu - \nu) \le 1 \text{ for all } \mu \in \Pcal, \nu \in \ba_+\}$, which certainly contains $\Ecal_b$. Conversely, suppose $f$ belongs to the right-hand side. Taking $\nu = 0$ we see that $\int f d\mu \le 1$ for all $\mu \in \Pcal$. Moreover, for any $\omega \in \Omega$ we may take an arbitrary $\mu \in \Pcal$ and set $\nu = \mu + t \delta_\omega$ for $t > 0$ to get $f(\omega) \ge - t^{-1}$. Sending $t$ to infinity shows that $f \in \Lcal_+$, and hence that $f \in \Ecal_b$. This establishes \eqref{eq_Eb_ba_rep}.
Next, the definition \eqref{eq_Peff_ba} states that $\Peff^\ba = \ba_+ \cap \Ecal_b^\circ$. Together with \eqref{eq_Eb_ba_rep}, this yields
\[
\Peff^\ba = \ba_+ \cap (\Pcal - \ba_+)^{\circ \circ}.
\]
Since $\Pcal - \ba_+$ contains zero, an application of the bipolar theorem completes the proof.
\end{proof}

\begin{lemma}\label{L_Peff_ba}
$\Peff^\ba$ is the solid hull of $\cconv(\Pcal)$,
i.e.,
\begin{equation} \label{eq_Peff_ba_sol_cconv_P}
\Peff^\ba = \sol \cconv(\Pcal).
\end{equation}
As a consequence, the set of maximal elements of $\Peff^\ba$  precisely equals $\cconv(\Pcal)$, and every $\mu \in \Peff^\ba$ is setwise dominated by some $\mu' \in \cconv(\P)$.
\end{lemma}

\begin{proof}
We prove \eqref{eq_Peff_ba_sol_cconv_P}; the remaining statement is then immediate.
In view of Lemma~\ref{L_Peff_ba_rep}, it suffices to show that
\begin{equation} \label{eq_Peff_ba_inclusions}
\ba_+ \cap \cconv(\Pcal - \ba_+) 
\subset \sol \cconv(\Pcal)
\subset \Peff^\ba.
\end{equation}
We start with the first inclusion, and let $\eta$ be an element of the left-hand side. 
Then $\eta \in \ba_+$ and there is a net $\eta_\alpha = \mu_\alpha - \nu_\alpha$ with $\mu_\alpha \in \conv(\Pcal)$ and $\nu_\alpha \in \ba_+$ such that $\eta_\alpha \to \eta$. 
Because each $\mu_\alpha$ belongs to $\ba_1$, which is weak-$*$ compact by the Banach--Alaoglu theorem, on passing to a subnet we have $\mu_\alpha \to \mu$ for some $\mu \in \ba_1$. 
In fact, $\mu \in \cconv(\Pcal)$. 
We also get $\nu_\alpha = \mu_\alpha - \eta_\alpha \to \nu = \mu - \eta$ which, like each $\nu_\alpha$, belongs to the weak-$*$ closed set $\ba_+$. 
In summary, we have shown that $0 \le \eta \le \mu$ with $\mu \in \cconv(\Pcal)$. That is, $\eta \in \sol \cconv(\Pcal)$.

We now prove the second inclusion in \eqref{eq_Peff_ba_inclusions}. It is clear from the definition \eqref{eq_Peff_ba} that $\Peff^\ba$ is convex, weak-$*$ closed, and contains $\Pcal$. Therefore it also contains $\cconv(\Pcal)$. It is also clear that $\Peff^\ba$ is solid, so it must contain $\sol \cconv(\Pcal)$ as well. 
This completes the proof of \eqref{eq_Peff_ba_inclusions}, and hence of \eqref{eq_Peff_ba_sol_cconv_P}.
\end{proof}

We now turn to the second part of Theorem~\ref{T_Peff}; this is the content of the following lemma.

\begin{lemma}
Every maximal element of $\Peff$ is the countably additive part of some element of $\cconv(\Pcal)$.
\end{lemma}

\begin{proof}
Let $\mu' \in \Peff$ be a maximal element.
Since $\Peff \subset \Peff^\ba$, the representation \eqref{eq_Peff_ba_sol_cconv_P} of Lemma~\ref{L_Peff_ba}, which we just established, implies that $\mu' \le \mu''$ for some $\mu'' \in \cconv(\Pcal)$. Let $\mu'' = \mu''_c + \mu''_p$ be the Yosida--Hewitt decomposition of $\mu''$ into its countably additive and purely finitely additive parts; see \citet[Theorem~1.24]{MR45194}.
The countably additive part $\mu''_c$ is the largest countably additive measure that is dominated by $\mu''$. Indeed, if $\nu \le \mu''$ is countably additive, then so is $(\nu - \mu''_c)^+ \le \mu''_p$, implying that $(\nu - \mu''_c)^+ = 0$ since $\mu''_p$ is purely finitely additive, and hence $\nu \le \mu''_c$.
We apply this with $\nu = \mu'$ to get $\mu' \le \mu''_c$. Moreover, because $\mu''_c$ is countably additive and dominated by $\mu''$, it belongs to $\Peff^\ba$ and hence, being countably additive, to $\Peff$.
But $\mu'$ is maximal in $\Peff$, so we must actually have $\mu' = \mu''_c$.
\end{proof}

{\color{black}
\section{Practical implications} \label{app:practical-implications}

\subsection{Contaminated null against simple alternative}
Consider testing a \emph{contaminated} null distribution against a simple alternative $\Qcal = \{\nu_0\}$.
The null hypothesis is given by $\Pcal = \{(1-\varepsilon)\mu_0 + \varepsilon \eta \colon \eta \in \Ccal\}$, where $\mu_0$ is a simple null distribution, $\Ccal$ is a convex set of contamination laws, and $\varepsilon \in (0,1)$ is the fraction of the population that comes from such a law.
The classical Kraft--Le~Cam theorem, if it applied, would suggest that the minimax risk is given by $1 - \dist_\TV(\Pcal, \Qcal)$. (No convex hulls are needed since both $\Pcal$ and $\Qcal$ are already convex.)
However, depending on the contamination model, this may give an incorrect answer.
For example, suppose $\mu_0$ and $\nu_0$ are both atomless, but $\Ccal$ is the convex hull of all Dirac point masses.
Then, because $\mu_0$ and $\nu_0$ are both singular with respect to all $\eta \in \Ccal$, we have in view of \eqref{eq:TV-dist-formula-proof-section} that
\[
\dist_\TV(\Pcal, \Qcal) = \frac12 \|(1-\varepsilon)\mu_0 - \nu_0\|_\TV + \frac{\varepsilon}{2}.
\]
However, thanks to Proposition~\ref{P:2}, the weak-$*$ closure of $\Ccal$ contains all of $\Mcal_1$, so $\cconv(\Pcal) \supset \{(1-\varepsilon)\mu_0+\varepsilon \eta \colon \eta \in \Mcal_1\}$.
In particular, by considering $\eta = \nu_0$ one finds
\[
\dist_\TV(\cconv(\Pcal), \Qcal) \le \frac12 (1-\varepsilon) \|\mu_0 - \nu_0\|_\TV.
\]
Now, the triangle inequality applied with the decomposition $(1-\varepsilon)(\mu_0-\nu_0) = ((1-\varepsilon)\mu_0 - \nu_0) + \varepsilon \nu_0$ yields
\[
(1-\varepsilon) \|\mu_0 - \nu_0\|_\TV \le \|(1-\varepsilon)\mu_0 - \nu_0\|_\TV + \varepsilon,
\]
and the inequality is always strict. Indeed, for any $\mu \in \Mcal$ and $\nu \in \Mcal_+$, one can only have $\|\mu + \nu\|_\TV = \|\mu\|_\TV + \|\nu\|_\TV$ if there is a $\nu$-full measure set $\Omega'$ such that $\mu(\cdot \cap \Omega')$ is nonnegative.
With $\mu = (1-\varepsilon)\mu_0 - \nu_0$ and $\nu = \varepsilon \nu_0$, this would yield $0 \le \mu(\Omega') = (1-\varepsilon) \mu_0(\Omega') - 1 \le -\varepsilon$, a contradiction.
Thus
\[
\dist_\TV(\cconv(\Pcal), \Qcal) < \dist_\TV(\Pcal, \Qcal),
\]
and the gap can be as large as $\varepsilon$ (for instance if $\mu_0 = \nu_0$).
Thus, for this choice of contamination model, the classical theory will underestimate the minimax risk.
On the other hand, for other choices this issue does not arise, such as when $\mu_0$, $\nu_0$, and $\Ccal$ are all dominated by a reference measure.
This illustrates the need to exercise caution when choosing the model.

\subsection{Quantized alternative against simple null}

Let $\Omega = [0,1)$ regarded as a circle (``mod~1'').
We consider the simple null $\Pcal = \{\mu_0\}$ where $\mu_0$ is the uniform distribution.
The alternative is that the data comes from some $\nu_0$, but is only observed up to a known precision $1/m$ and with unknown bias $\theta$ for some $m \in \N$ and $\theta \in [0,1)$.
More precisely, for each $\theta$, let $q_\theta$ be the \emph{quantization map} that rounds any $x \in [0,1)$ down to the nearest point in the grid $\{\theta + k/m \colon k=0,\ldots,m-1\}$.
The alternative is then defined as the set of pushforwards
\[
\Qcal = \{(q_\theta)_* \nu_0 \colon \theta \in [0,1)\}.
\]
Every element of $\conv(\Qcal)$ is singular with respect to $\mu_0$, so the TV distance is one, and classical Kraft--Le~Cam suggests the risk is zero.
However, it can be shown that $\cconv(\Qcal)$ actually contains, in particular, all mixtures $\nu_h = \int_{[0,1)} (q_\theta)_* \nu_0 \, h(\theta) d\theta$ where $h(\theta)$ is a density on $[0,1)$.
A calculation shows that any such mixture has a Lebesgue density given by
\[
f_h(x) = \nu_0\left(\left[x, x+\frac1m\right)\right) \sum_{k=0}^{m-1} h\left(x - \frac{k}{m}\right).
\]
The TV distance to $\mu_0$ equals $\frac{1}{2} \int_{[0,1)} |f_h(x)-1| dx$, which can be smaller than one. 
Here are two cases where this happens:
\begin{itemize}
\item Taking $h \equiv 1$, the TV distance $\dist_\TV(\mu_0,\nu_h)$ tends to $\dist_\TV(\mu_0, \nu_0)$ in the infinite-precision limit $m \to \infty$. Thus for large $m$, the risk is approximately bounded below by the simple-vs-simple risk $1-\dist_\TV(\mu_0, \nu_0)$.

\item Suppose $\nu_0$ is $(1/m)$-periodic in the sense that $\nu_0([x, x+1/m))$ is constant in $x$; this is the case, for instance, if $\nu_0$ has a $(1/m)$-periodic density.
Then with $h \equiv 1$, $f_h$ is constant and $\nu_h$ is the uniform density. Thus $\dist_\TV(\mu_0,\nu_h) = 0$ and the risk is equal to one in this (rather special) case.
\end{itemize}

\subsection{Minimax risk via convex programs}

We now present a method by which the minimax risk can, in principle, be computed via a sequence of convex programs.
Recall that $\Pi$ denotes the set of finite measurable partitions of $\Omega$, and that $\mu_\pi$ is given by \eqref{eq:mu-pi} for any measure $\mu$.
For each $\pi \in \Pi$, consider the subsets of the $|\pi|$-simplex given by
\[
C_\pi = \textnormal{cl} \{\mu_\pi \colon \mu \in \conv(\Pcal)\},
\quad
D_\pi = \textnormal{cl} \{\nu_\pi \colon \nu \in \conv(\Qcal)\},
\]
where ``$\textnormal{cl}$'' refers to the Euclidean closure.

\begin{corollary}
\label{cor:finite-partition}
We have
\begin{equation}
\label{eq:finite-partition}
1-R(\mathcal P,\mathcal Q)
=
\sup_{\pi\in\Pi}
\min_{\substack{p\in C_\pi \\ q\in D_\pi}}
\frac12\lVert p-q\rVert_1.
\end{equation}
Consequently, for every $\varepsilon\geq0$, the minimax risk is below $1-\varepsilon$
if
and only if there is a finite measurable partition $\pi$ for which
$d_{\mathrm{TV}}(C_\pi,D_\pi)>\varepsilon$.  In that event, there is a test with risk below $1-\varepsilon$ which is constant on the sets of $\pi$. 
\end{corollary}

\begin{proof}
Note that $|R(\phi)-R(\psi)|
\leq 2\|\phi-\psi\|_\infty$ for all $\phi,\psi \in \Phi$.
Since every test can be approximated uniformly by a test that is constant on a finite measurable partition, it suffices to compute the minimax risk over such tests.
Fix $\pi=\{A_1,\ldots,A_k\}$ and parameterize these tests by $a\in[0,1]^k$ via $\phi_a=\sum_{i=1}^k a_i\mathbf 1_{A_i}$.
The finite-dimensional minimax theorem gives
\begin{align*}
\sup_{a\in[0,1]^k} & (1-R(\phi_a))
=\sup_{a\in[0,1]^k}
\min_{\substack{p\in C_\pi \\ q\in D_\pi}} a^\top(q-p)\\
&= \min_{\substack{p\in C_\pi \\ q\in D_\pi}}
\sup_{a\in[0,1]^k}a^\top(q-p)
=\min_{\substack{p\in C_\pi \\ q\in D_\pi}}
\frac12\lVert p-q\rVert_1,
\end{align*}
where the last equality follows from
$\sum_i(q_i-p_i)=0$.  Taking the supremum over $\pi$ proves
\eqref{eq:finite-partition}.
\end{proof}

Combining \eqref{eq:finite-partition} with Theorem~\ref{T_kraft} yields
\[
 \dist_\TV(
 \cconv(\mathcal P),
 \cconv(\mathcal Q)
 )
=
\sup_{\pi\in\Pi}
 \dist_\TV(C_\pi,D_\pi).
\]
If $\pi'$ refines $\pi$, then $\dist_\TV(C_{\pi'},D_{\pi'})
\geq \dist_\TV(C_\pi,D_\pi)$,
so the finite-partition values form an increasing net under
refinement.
}

\end{document}